\documentclass[11pt,a4paper]{amsart}
\usepackage{amsthm,amsfonts,amsmath, amssymb,graphicx,float,enumerate,iwona}

\usepackage[numbers,sort&compress]{natbib}
\usepackage[lmargin=33mm,rmargin=33mm,tmargin=33mm,bmargin=32mm]{geometry}
\usepackage[colorlinks=true,citecolor=black,linkcolor=black,urlcolor=blue]{hyperref}

\renewcommand{\baselinestretch}{1.25}
\newcommand{\journalarxiv}[2]{#2}

\DeclareMathOperator{\dist}{dist}
\DeclareMathOperator{\diam}{diam}

\newcommand{\msn}[1]{MR:\,\href{http://www.ams.org/mathscinet-getitem?mr=MR#1}{#1}}
\newcommand{\MSN}[2]{MR:\,\href{http://www.ams.org/mathscinet-getitem?mr=MR#1}{#1}}
\newcommand{\doi}[1]{doi:\,\href{http://dx.doi.org/#1}{#1}}

\newcommand{\Zbl}[1]{Zbl:\,\href{http://www.zentralblatt-math.org/zmath/en/search/?q=an:#1}{#1}}

\usepackage{verbatim}

\theoremstyle{plain}
\newtheorem{theorem}{Theorem}
\newtheorem{lemma}[theorem]{Lemma}
\newtheorem{corollary}[theorem]{Corollary}
\newtheorem{proposition}[theorem]{Proposition}

\theoremstyle{definition}

\newcommand{\thmlabel}[1]{\label{thm:#1}}
\newcommand{\thmref}[1]{Theorem~\ref{thm:#1}}
\newcommand{\lemlabel}[1]{\label{lem:#1}}
\newcommand{\lemref}[1]{Lemma~\ref{lem:#1}}
\newcommand{\twolemref}[2]{Lemmas~\ref{lem:#1} and \ref{lem:#2}}

\newcommand{\figlabel}[1]{\label{fig:#1}}
\newcommand{\figref}[1]{Figure~\ref{fig:#1}}

\newcommand{\seclabel}[1]{\label{sec:#1}}
\newcommand{\secref}[1]{Section~\ref{sec:#1}}
\newcommand{\corlabel}[1]{\label{cor:#1}}
\newcommand{\corref}[1]{Corollary~\ref{cor:#1}}
\newcommand{\proplabel}[1]{\label{prop:#1}}
\newcommand{\propref}[1]{Proposition~\ref{prop:#1}}
\newcommand{\twopropref}[2]{Propositions~\ref{prop:#1} and \ref{prop:#2}}

\newcommand{\GG}{\ensuremath{\mathcal{G}}}

\newcommand{\D}{\Delta}

\newcommand{\BB}{\overrightarrow{B}}

\newcommand{\ceil}[1]{\ensuremath{\protect\lceil{#1}\rceil}}
\newcommand{\floor}[1]{\ensuremath{\protect\lfloor{#1}\rfloor}}

\newcommand{\quarter}{\ensuremath{\tfrac14}}
\newcommand{\half}{\ensuremath{\tfrac12}}

\renewcommand{\geq}{\geqslant}
\renewcommand{\leq}{\leqslant}

\begin{document}

\title[The degree-diameter problem for sparse graph classes]{The degree-diameter problem\\ for sparse graph classes}

\author{Guillermo Pineda-Villavicencio}\thanks{Centre for Informatics and Applied Optimisation, Federation University Australia, Ballarat, Australia (\texttt{work@guillermo.com.au}).  Research supported by a postdoctoral fellowship funded by the Skirball Foundation via the Center for Advanced Studies in Mathematics at Ben-Gurion University of the Negev, and by an ISF grant. Corresponding author.}

\author{David R. Wood}\thanks{School of Mathematical Sciences, Monash  University, Melbourne, Australia
    (\texttt{david.wood@monash.edu}). Research supported by the Australian Research Council. }

\subjclass[2000]{}

\date{\today}

\begin{abstract} 
The degree-diameter problem asks for the maximum number of vertices in a graph with maximum degree $\Delta$ and diameter $k$. For fixed $k$, the answer is $\Theta(\Delta^k)$. We consider the degree-diameter problem for particular classes of sparse graphs, and establish the following results. For graphs of bounded average degree the answer is $\Theta(\Delta^{k-1})$, and for graphs of bounded arboricity the answer is $\Theta(\Delta^{\floor{k/2}})$, in both cases  for fixed $k$. For graphs of given treewidth, we determine the the maximum number of vertices up to a constant factor. More precise bounds are given for graphs of given treewidth, graphs embeddable on a given surface, and apex-minor-free graphs. 
\end{abstract}

\maketitle

\section{Introduction} 
\seclabel{Intro}

Let $N(\Delta,k)$ be the maximum number of vertices in a graph with maximum degree at most $\Delta$ and diameter at most $k$. Determining $N(\Delta,k)$ is called the \emph{degree-diameter} problem and is widely studied, especially motivated by questions in network design; see \citep{MS-EJC05} for a survey. Obviously, $N(\Delta,k)$ is at most the number of vertices at distance at most $k$ from a fixed vertex. For $\D\geq3$ (which we implicitly assume),  it follows that
$$N(\Delta,k)\;\leq\; M(\D,k)\;:=\;1+\Delta\sum_{i=1}^{k-1}(\Delta-1)^i\;=\;\frac{\D(\D-1)^k-2}{\D-2} 
\enspace.$$
This  inequality is called the  Moore bound. 
The best lower bound is
$$N(\Delta,k)\geq f(k)\, \Delta^k\enspace,$$
for some function $f$. For example, the de Bruijn graph shows that
$N(\Delta,k)\geq\big(\tfrac{\Delta}{2}\big)^k$; see \lemref{deBruijn}. 
\citet{CG05} established the best known asymptotic bound of 
$N(\Delta,k)\geq\big(\tfrac{\Delta}{1.59}\big)^k$ for sufficiently large $\Delta$.

For a class of graphs \GG, let $N(\Delta,k,\GG)$ be the maximum number
of vertices in a graph in \GG\ with maximum degree at most $\Delta$ and
diameter at most $k$.  We consider  $N(\Delta,k,\GG)$ for some
particular classes \GG\ of sparse graphs, focusing on the case of  small diameter $k$, and large maximum degree $\Delta$.  We prove lower and upper bounds on $N(\Delta,k,\GG)$ of the form 
\begin{equation}
\label{MainEquation}
f(k)\, \Delta^{ g(k) }\end{equation} for some functions $f$ and $g$. Since $k$ is assumed to be small compared to $\Delta$, the most important term in such a bound is $g(k)$. Thus our focus is on $g(k)$ with $f(k)$ a secondary concern. 

We first state two straightforward examples, namely bipartite graphs and trees. The maximum number of vertices in a bipartite graph with maximum degree $\Delta$ and diameter $k$ is  $f(k)\,\Delta^{k-1}$ for some function $f$; see references \cite[Section 2.4.4]{MS-EJC05} and \citep{Biggs74,Delorme85}.
And for trees, it is easily seen that the maximum number of vertices is within a constant factor of $(\Delta-1)^{\floor{k/2}}$, which is a big improvement over the unrestricted bound of $\Delta^k$. Some of the results in this paper can be thought of as generalisations of this observation. 

In what follows we initially consider broadly defined classes of sparse graphs, moving progressively towards more specific classes. The following table summarises our current knowledge, where results in this paper are in bold. 

\medskip
\begin{center}
\begin{tabular}{llr}
\hline
graph class & diameter $k$ & max.\ number of vertices\\\hline
general & & $f(k)\,\Delta^k$ \\
{\bf 3-colourable} & $k\geq2$ & $f(k)\,\Delta^k$\\
\bf triangle-free 3-colourable & $k\geq4$ & $f(k)\,\Delta^k$\\
bipartite & & $f(k)\,\Delta^{k-1}$\\
\bf average degree $d$ & & $f(k)\,d\Delta^{k-1}$\\
\bf arboricity $b$ &&  $f(k,b)\,\Delta^{\floor{k/2}}$\\
\bf treewidth $t$ & odd $k$ & $ct\,(\Delta-1)^{(k-1)/2}$\\
\bf treewidth $t$ & even $k$ & $c\sqrt{t}\,(\Delta-1)^{k/2}$\\
\bf Euler genus $g$ & odd $k$ & $\leq c(g+1)k\,(\Delta-1)^{(k-1)/2}$\\
\bf Euler genus $g$ & even $k$ & $\leq c\sqrt{(g+1)k}\,(\Delta-1)^{k/2}$\\
trees & & $c\,\Delta^{\floor{k/2}}$\\
\hline
\end{tabular}
\end{center}
\medskip

First consider the class of graphs with average degree $d$. In this case, we prove that the maximum number of vertices is $f(k)\,d\Delta^{k-1}$ for some function $f$ (see \secref{Degree}). This shows that by assuming bounded average degree we obtain a modest improvement over the standard bound of $(\Delta-1)^k$. A much more substantial improvement is obtained by considering arboricity.

The \emph{arboricity} of a graph $G$ is the minimum number of spanning
forests whose union is $G$.
\citet{NW-JLMS64} proved that the arboricity of $G$ equals
\begin{equation}
  \label{NashWilliams}
  \max_{H\subseteq G}\,\big\lceil\tfrac{|E(H)|}{ |V(H)|-1}\big\rceil\enspace,
\end{equation}
where the maximum is taken over all subgraphs $H$ of $G$.
For example, it follows from Euler's formula that every planar graph
has arboricity at most $3$, and every graph with
Euler genus $g$ has arboricity at most $O(\sqrt{g})$. 
More generally, every graph that excludes a fixed minor has bounded arboricity. 
Note that $\delta\leq d\leq 2b$ for every graph with minimum degree $\delta$,
average degree $d$, and arboricity $b$. Arboricity is a more refined
measure than average degree, in the sense that a graph has bounded
arboricity if and only if every subgraph has bounded average degree.

We prove that for a graph with arboricity $b$ the maximum number of vertices is $f(b,k)\,\Delta^{\floor{k/2}}$ for some function $f$ (see \secref{Arboricity}). Thus by moving from bounded average degree to bounded arboricity the $g(k)$ term discussed above is reduced from $k-1$ to $\floor{\frac{k}{2}}$. This result generalises the above-mentioned bound for trees, which have arboricity 1. The dependence on $b$ in $f$ can be reduced by making more restrictive assumptions about the graph. 

For example, treewidth is a parameter that measures how tree-like a given graph is. The \emph{treewidth} of a graph $G$ can be defined to be the minimum integer $t$ such that $G$ is a spanning subgraph of a chordal\footnote{A graph is \emph{chordal} if every induced cycle is a triangle.} graph with no $(t+2)$-clique. For example, trees are exactly the connected graphs with treewidth 1. See \citep{Reed97,Bodlaender-TCS98} for background on treewidth. Since the arboricity of a graph is at most its treewidth, bounded treewidth is indeed a more restrictive assumption than bounded arboricity. We prove that the maximum number of vertices in a graph with treewidth $t$ is within a constant factor of $t(\Delta-1)^{(k-1)/2}$ if $k$ is odd, and of $\sqrt{t}(\Delta-1)^{k/2}$  if $k$ is even (and $\Delta$ is large). These results immediately imply the best known bounds for graphs of given Euler genus\footnote{A \emph{surface} is a non-null compact connected 2-manifold without boundary. Every surface is homeomorphic to the sphere with $h$ handles or the sphere with $c$ cross-caps.  The sphere with $h$ handles has \emph{Euler genus} $2h$, and the sphere with $c$ cross-caps has \emph{Euler genus} $c$. The \emph{Euler genus} of a graph $G$ is the minimum Euler genus of a surface in which $G$ embeds. See the monograph by \citet{MoharThom} for background on graphs embedded in surfaces.}, and new bounds for apex-minor-free graphs. All these results are presented in \secref{Separators}.


Our results in \secref{3Colour} are of a different nature. There, we describe (non-sparse) graph classes for which the maximum number of vertices is not much different from the unrestricted case. In particular, we prove that for $k\geq2$, there are 3-colourable graphs with $f(k)\,\Delta^k$ vertices, and for 
for $k\geq4$, there are triangle-free 3-colourable graphs with $f(k)\,\Delta^k$ vertices. These results are in contrast to the bipartite case, in which $f(k)\,\Delta^{k-1}$ is the answer. 

All undefined terminology and notation is in reference \cite{D05}.


\section{Basic Constructions}

This section gives some graph constructions that will later be used for proving lower bounds on $N(\Delta,k,\GG)$.

\begin{lemma}
  \lemlabel{deBruijn}
  For all integers $r\geq 1$ and $k\geq 1$ the de Bruijn graph $B(r,k)$
has $r^k$ vertices, maximum degree at most $2r$, and diameter
  $k$. Moreover, for $k\geq 2$, there are sets $B_1,\dots,B_{r^{k-1}}$
  of vertices in $B(r,k)$, each containing $2r-2$ or $2r$ vertices,
  such that each vertex of $B(r,k)$ is in exactly two of the $B_i$,
  and the endpoints of each edge of $B(r,k)$ are in some $B_i$.
\end{lemma}

\begin{proof}
  In what follows, a \emph{digraph} is a directed graph possibly with
  loops and possibly with arcs in opposite directions between two
  vertices. A digraph is \emph{$r$-inout-regular} if each vertex has
  indegree $r$ and outdegree $r$ (where a loop at $v$ counts in the
  indegree and the outdegree of $v$). A digraph has \emph{strong
    diameter} $k$ if for all (not necessarily distinct) vertices $v$
  and $w$ there is a directed walk from $v$ to $w$ of length exactly
  $k$.

  Let $\BB(r,k)$ be the de Bruijn digraph \citep{deBruijn,Good46},
  which has $r^k$ vertices, is $r$-inout-regular, and has diameter
  $k$.  Fiol et al.~\cite[Sec.~IV]{FioYebAle84}, and \citet{ZL87} showed that the $\BB(r,k)$ can be constructed
  recursively as a line digraph, as we now explain. If $G$ is a
  digraph with arc set $A(G)$, then the \emph{line digraph} $L(G)$ has
  vertex set $A(G)$, where $(uv,vw)$ is an arc of $L(G)$ for all distinct arcs
  $uv,vw\in A(G)$.
  
   Let $\BB(r,1)$ be the $r$-vertex digraph in which every arc is
  present (including loops). Now recursively define
  $\BB(r,k):=L(\BB(r,k-1))$. The digraph $\BB(r,k)$ is $r$-inout-regular and has strong diameter $k$; see \cite[Sec.~IV]{FioYebAle84}.

  Define $B(r,k)$ to be the undirected graph that underlies
  $\BB(r,k)$ (ignoring loops, and replacing bidirectional arcs by a
  single edge). Then $B(r,k)$ has $r^k$ vertices, has maximum degree
  at most $2r$, and has (undirected) diameter $k$ (since loops can be
  ignored in shortest paths).

  It remains to prove the final claim of the lemma, where $k\geq 2$.
  For each vertex $v$ of $\BB(r,k-1)$, let $B_v$ be the set of
  vertices of $B(r,k)$ that correspond to non-loop arcs incident with
  $v$ in $\BB(r,k-1)$. Thus $|B_v|$ equals $2r-2$ or $2r$ depending on
  whether there is a loop at $v$ in $\BB(r,k-1)$.  Each vertex of
  $B(r,k)$ corresponding to an arc $vw$ of $\BB(r,k-1)$ is in exactly
  two of these sets, namely $B_v$ and $B_w$.  The endpoints of each
  edge of $B(r,k)$ corresponding to a path $uv,vw$ of $\BB(r,k-1)$ are
  both in $B_v$. These $r^{k-1}$ sets, one for each vertex of
  $\BB(r,k-1)$, define the desired sets in $B(r,k)$.
\end{proof}

The next two lemmas will be useful later. 

\begin{lemma}
  \lemlabel{LineGraph}
  For every integer $q\geq 1$ there is a $(2q-2)$-regular graph $L$
  with $\binom{q+1}{2}$ vertices, containing cliques
  $L_1,\dots,L_{q+1}$ each of order $q$, such that each vertex in $L$
  is in exactly two of the $L_i$, and $L_i\cap L_j\neq\emptyset$ for
  all $i,j\in[1,q+1]$.
\end{lemma}

\begin{proof}
  Let $L$ be the line graph of the complete graph $K_{q+1}$. That is,
  $V(L):=\{\{i,j\}:1\leq i,j\leq q+1,\, i\neq j\}$, where
  $L_i:=\{\{i,j\}:1\leq j\leq q+1,\, i\neq j\}$ is a clique for each
  $i\in[1,q+1]$. The claimed properties are immediate.
\end{proof}

\begin{lemma}
  \lemlabel{BipartiteLemma}
  For all integers $p\geq1$ and $q\geq 1$ and $m\leq (q+1)p$ there is
  a bipartite graph $T$ with bipartition $C,D$, such that $C$ consists
  of $m$ vertices each with degree $q$, and $D$ consists of
  $\binom{q+1}{2}$ vertices each with degree at most $2p$, and every
  pair of vertices in $C$ have a common neighbour in $D$.
\end{lemma}

\begin{proof}
  By \lemref{LineGraph}, there is a set $D$ of size
  $\binom{q+1}{2}$, containing subsets $D_1,\dots,D_{q+1}$ each of
  size $q$, such that each element of $D$ is in exactly two of the
  $D_i$, and $D_i\cap D_j\neq\emptyset$ for all $i,j\in[1,q+1]$.

  Let $T$ be the graph with vertex set $C\cup D$, where $C$ is defined
  as follows. For each $i\in[1,q+1]$ add a set $C_i$ of $p$ vertices
  to $C$, each adjacent to every vertex in $D_i$. Since $|D_i|=q$,
  each vertex in $C$ has degree $q$. Since each element of $D$ is in
  exactly two of the $D_i$, each vertex in $D$ has degree $2p$.

  Consider two vertices $v,w\in C$. Say $v\in C_i$ and $w\in C_j$. Let
  $x$ be a vertex in $D_i\cap D_j$. Then $x$ is a common neighbour of
  $v$ and $w$ in $G$. 

We have proved that $T$ has the desired properties in the case that
  $m=(q+1)p$. Finally, delete $(q+1)p-m$ vertices from $C$,
  and the obtained graph has the desired properties.
\end{proof}

%
%
  
\section{Average Degree}
\seclabel{Degree}

This section presents bounds on the maximum number of vertices in a graph with given average degree. For fixed diameter, the upper and lower bounds are within a constant factor. We have the following rough upper bound for graphs of given minimum degree.

\begin{proposition}
  \proplabel{MinimumDegreeUpperBound}
  Every graph with minimum degree $\delta$, maximum degree $\Delta$ and
  diameter $k$ has at most $2\delta(\Delta-1)^{k-1}+1$ vertices.
\end{proposition}

\begin{proof}
  Let $v$ be a vertex of degree $\delta$.  For $0\leq i\leq k$, let $n_i$
  be the number of vertices at distance $i$ from $v$. Thus $n_0=1$ and
  $n_i\leq \delta(\Delta-1)^{i-1}$ for all $i\geq 1$.  In total,
  $n=\sum_{i=0}^kn_i\leq 1+\sum_{i=1}^k\delta(\Delta-1)^{i-1}= 1 +
  \delta\,\frac{(\Delta-1)^k-1}{\Delta-2} \leq 1 + 2\delta(\Delta-1)^{k-1}$.
\end{proof}

Since minimum degree is at most average degree, we have the following
corollary.

\begin{corollary}
  \corlabel{AverageDegreeUpperBound}
  Every graph with average degree $d$, maximum degree $\Delta$ and
  diameter $k$ has at most $2d(\Delta-1)^{k-1}+1$ vertices.
\end{corollary}

The following is the main result of this section; it says that \corref{AverageDegreeUpperBound} is within a constant factor of optimal for fixed $k$.




\begin{proposition}
  \proplabel{AverageDegreeLowerBound}
  For all integers $d\geq 4$ and $k\geq 3$ and $\Delta\geq 2d$ there is a
  graph with average degree at most $d$, maximum degree at most
  $\Delta$, diameter at most $k$, and at least
  $\frac{d}{8}\floor{\frac{\Delta}{4}}^{k-1}$ vertices.
\end{proposition}

\begin{proof}
  Let $r:=\floor{\frac{\Delta}{4}}$. Let $q:=\floor{\frac{d}{4}}\geq
  2$.  Let $p:=\floor{\frac{\Delta}{2}}-r-q+1$. Note that $d\geq 4q$ and
  $4p\geq \Delta-4q \geq \frac{\Delta}{2}$.

  Let $B:=B(r,k-2)$ be the graph from \lemref{deBruijn} with
  maximum degree at most $2r$, diameter $k-2$, and $r^{k-2}$ vertices.

  Let $L$ be the $(2q-2)$-regular graph from \lemref{LineGraph}
  with $\binom{q+1}{2}$ vertices, containing cliques
  $L_1,\dots,L_{q+1}$ each of order $q$, such that each vertex in $L$
  is in exactly two of the $L_i$, and $L_i\cap L_j\neq\emptyset$ for
  all $i,j\in[1,q+1]$.

  Let $H$ be the cartesian product graph $L\,\square\, B$.  Note that
  $H$ has $\binom{q+1}{2}\,r^{k-2}$ vertices and has maximum degree at
  most $2q-2+2r$.  For $i\in[1,q+1]$ and $v\in V(B)$, let $X_{i,v}$
  be the clique $\{(x,v):x\in L_i\}$ in $H$.  Since each vertex in $L$
  is in exactly two of the $L_i$, each vertex in $H$ is in exactly two
  of the $X_{i,v}$.

  Let $G$ be the graph obtained from $H$ as follows: for $i\in[1,q+1]$
  and $v\in V(B)$, add an independent set $Y_{i,v}$ of $p$ vertices to
  $G$ completely adjacent to $X_{i,v}$; that is, every vertex in
  $Y_{i,v}$ is adjacent to every vertex in $X_{i,v}$. We now prove
  that $G$ has the claimed properties.

  The number of vertices in $G$ is
$$|V(G)|
\geq \sum_{i,v}|Y_{i,v}| =(q+1)r^{k-2} p \geq
\tfrac{d}{4}\floor{\tfrac{\Delta}{4}}^{k-2} \tfrac{\Delta}{8} \geq
\tfrac{d}{8}\floor{\tfrac{\Delta}{4}}^{k-1} \enspace.$$

To determine the diameter of $G$, let $\alpha$ and $\beta$ be vertices
in $G$. Say $\alpha\in X_{i,v}\cup Y_{i,v}$ and $\beta\in X_{j,w}\cup
Y_{j,w}$. Let $x$ be a vertex in $L_i\cap L_j$. Let
$v=y_1,\dots,y_{\ell}=w$ be a path of length at most $k-2$ in
$B$. Then $\alpha,(x,y_1),(x,y_2),\dots,(x,y_\ell),\beta$ is path of
length at most $k$ in $G$. Hence $G$ has diameter at most $k$.

Consider the maximum degree of $G$. Each vertex in some set $Y_{i,v}$
has degree $|X_{i,v}|=|L_i|=q\leq\Delta$. Each vertex in some set
$X_{i,v}$ has degree $2q-2+2r+2p\leq \Delta$. Thus $G$ has maximum
degree at most $\Delta$.

It remains to prove that the average degree of $G$ is at most $d$.
There are $|V(H)|= \binom{q+1}{2}\,r^{k-2}$ vertices of degree at most
$\Delta$, and there are $(q+1) r^{k-2}p$ vertices of degree $q$. Thus
the average degree is at most
$$
\frac{\binom{q+1}{2}\,r^{k-2}\cdot \Delta\;+\; (q+1)
  r^{k-2}pq}{\binom{q+1}{2}\,r^{k-2} \;+\; (q+1) r^{k-2}p} =
\frac{\frac{q}{2} \Delta\;+\; pq}{\frac{q}{2} \;+\; p}$$ Hence it
suffices to prove that $q\Delta + 2pq \leq (q + 2p)d$.  Since
$\Delta\geq 2d$ and $d\geq 4q$,
 $$d\Delta 
 = \tfrac{d\Delta}{2} + \tfrac{d\Delta}{2} \geq \tfrac{d\Delta}{2} +
 d^2 \geq 2q\Delta + 4dq\enspace.$$ That is, $2d\Delta -
 2q\Delta  - 8qd \geq d\Delta- 4qd$.  Since $4p\geq \Delta-4q$
 and $8q^2\geq 0$,
$$8p(d-q)\geq 2(\Delta-4q)(d-q) 
= 2d\Delta-2q\Delta-8qd+8q^2 \geq d\Delta- 4qd \geq 4q\Delta-
4qd \enspace.$$ That is, $4pd+2qd\geq 2q\Delta + 4pq$, as
desired. Hence the average degree of $G$ is at most $d$.
\end{proof}

Note that for particular values of $k$ and $\Delta$, other graphs can
be used instead of the de Bruijn graph in the proof of
\propref{AverageDegreeLowerBound} to improve the constants in
our results; we omit all these details.

\section{Arboricity}
\seclabel{Arboricity}

This section proves that the maximum number of vertices in a graph
with arboricity $b$ is $f(b,k)\cdot\Delta^{\floor{k/2}}$ for some
function $f$. Reasonably tight lower and upper bounds on $f$ are
established. First we prove the upper bound.

\begin{theorem}
  \thmlabel{ArboricityUpperBound}
  For every graph $G$ with arboricity $b$, diameter $k$, and maximum
  degree $\Delta$,
$$|V(G)| \leq 4k (2b)^k \Delta^{\floor{k/2}}+1\enspace.$$
\end{theorem}

\begin{proof}
  Let $G_1,\dots,G_b$ be spanning forests of $G$ whose union is
  $G$. Orient the edges of each component of each $G_i$ towards a root
  vertex. Thus each vertex $v$ of $G$ has outdegree at most 1 in each
  $G_i$; therefore $v$ has outdegree at most $b$ in $G$.

  Consider an unordered pair of vertices $\{v,w\}$. Let $P$ be a
  shortest $vw$-path in $G$. Say $P$ has $\ell$ edges. Then $\ell\leq
  k$. An edge of $P$ oriented in the direction from $v$ to $w$ is
  called \emph{forward}. If at least $\ceil{\frac{\ell}{2}}$ of the
  edges in $P$ are forward, then charge the pair $\{v,w\}$ to $v$,
  otherwise charge $\{v,w\}$ to $w$.

  Consider a vertex $v$. If some pair $\{v,w\}$ is charged to $v$ then
  there is path of length $\ell$ from $v$ to $w$ with exactly $i$
  forward arcs, for some $i$ and $\ell$ with
  $\ceil{\frac{\ell}{2}}\leq i\leq \ell\leq k$. Since each vertex has
  outdegree at most $b$, the number of such paths is at most
  $\binom{\ell}{i} b^i\Delta^{\ell-i}$. Hence the number of pairs
  charged to $v$ is at most
  \begin{align*}
    \sum_{\ell=1}^k\sum_{i=\ceil{\ell/2}}^\ell\binom{\ell}{i}b^i\Delta^{\ell-i}
    \;\leq\;& k \sum_{i=\ceil{k/2}}^k\binom{k}{i}b^i\Delta^{k-i}
    \;=\;k \sum_{i=0}^{\floor{k/2}}\binom{k}{k-i}b^{k-i}\Delta^{i}\\
    \;\leq\;& k\,2^kb^k \sum_{i=0}^{\floor{k/2}}\Delta^{i} \;\leq\; 2k
    (2b)^k \Delta^{\floor{k/2}}\enspace.
  \end{align*}
  Hence, the total number of pairs, $\binom{n}{2}$, is at most $2k
  (2b)^k \Delta^{\floor{k/2}}n$. The result follows.
\end{proof}



We now show that the upper bound in \thmref{ArboricityUpperBound} 
is close to being best possible (for fixed $k$). 



\begin{theorem}
\thmlabel{ArboricityLowerBound}
  For all even integers $b\geq 2$ and $k\geq 4$ and $\Delta \geq b$,
  such that $\Delta\equiv 2\pmod{4}$ or $b\equiv 0\pmod{4}$, there is
  a graph $G$ with arboricity at most $b$, maximum degree at most
  $\Delta$, diameter at most $k$, and at least $\frac{8}{b^2}
  (\frac{b\Delta}{8})^{k/2}$ vertices.
\end{theorem}

\begin{proof}
  Let $q:=\frac{\Delta}{2}$ and $p:=\frac{b}{2}$ and
  $\ell:=\frac{k}{2}-1$. Then $q$, $p$ and $\ell$ are positive
  integers. Let $r:=\frac{(q+1)p}{2}$.  Then $r$ is a positive integer
  (since $\Delta\equiv 2\pmod{4}$ or $b\equiv 0\pmod{4}$).

  Let $B$ be the de Bruijn graph $B(r,\ell)$.  By
  \lemref{deBruijn}, $B$ has diameter $\ell$ and $r^\ell$
  vertices. Moreover, there are sets $B_1,\dots,B_{r^{\ell-1}}$ of
  vertices in $B$, each containing $2r-2$ or $2r$ vertices, such that
  each vertex of $B$ is in exactly two of the $B_i$, and the endpoints
  of each edge in $B$ are in some $B_i$. Let $r_i:=|B_i|$. Thus
  $r_i\leq 2r=(q+1)p$.

  By \lemref{BipartiteLemma}, for each $i\in [1,r^{\ell-1}]$ there
  is a bipartite graph $T_i$ with bipartition $B_i,D_i$, such that
  $B_i$ consists of $r_i$ vertices each with degree $q$, and
  $D_i$ consists of $\binom{q+1}{2}$ vertices each with degree at most
  $2p\leq b$, and each pair of vertices in $B_i$ have a common neighbour in
  $D_i$.

  Let $G$ be the bipartite graph with bipartition $V(B)\cup D$, where
  $D:=\cup_i D_i$ and the induced subgraph $G[B_i,D_i]$ is $T_i$.
  In $G$, each vertex in $V(B)$ has degree $2q\leq\Delta$, and each
  vertex in $D$ has degree at most $b\leq\Delta$.  Thus $G$ has maximum degree
  $\Delta$.  Assign each edge in $G$ one of $b$ colours, such that two
  edges receive distinct colours whenever they have an endpoint in $D$
  in common. Each colour class is a star forest. Hence $G$ has
  arboricity at most $b$.  Observe that $$|V(G)|\geq |D|= 
r^{\ell-1}\tbinom{q+1}{2} \geq
  (\tfrac{b\Delta}{8})^{\ell-1}\,\tfrac{\Delta^2}{8} =
 (\tfrac{b\Delta}{8})^{k/2-2}\,\tfrac{\Delta^2}{8} =
\tfrac{8}{b^2}  (\tfrac{b\Delta}{8})^{k/2}\enspace.
$$

It remains to prove that $G$ has diameter at most $k$.  Consider two
  vertices $v$ and $w$ in $G$.  If $v\in D_i$ then let $v'$ be a
  neighbour of $v$ in $B_i$.  If $v\in B_i$ then let $v'$ be $v$.  If
  $w\in D_j$ then let $w'$ be a neighbour of $w$ in $B_j$.  If $w\in
  B_j$ then let $w'$ be $w$.  In $B$, there is a $v'w'$-path $P$ of
  length at most $\ell$.  For each edge $xy$ in $P$, both $x$ and $y$
  are in some set $B_a$ (see \lemref{deBruijn}).  Since $x$ and $y$ have a common neighbour in
  $T_a$ (by \lemref{BipartiteLemma}), we can
  replace $xy$ in $P$ by a 2-edge path in $T_a$, to obtain a
  $v'w'$-path in $G$ of length at most $2\ell$.  Possibly adding the
  edges $vv'$ or $ww'$ gives a $vw$-path in $G$ of length at most
  $2\ell+2=k$. Hence $G$ has diameter at most $k$.
\end{proof}

Consider the case of diameter $2$ graphs with arboricity $b$. Every such graph has average degree less than $2b$, and thus 
has at most $4b\Delta$ vertices by \corref{AverageDegreeUpperBound}.  We now show that this upper bound is within a constant factor of optimal. (This result is not covered by \thmref{ArboricityLowerBound} which assumes $k\geq4$.) 



\begin{proposition}
  For all integers $b\geq1$ and even $\Delta\geq 4b$ there is a graph with diameter
  $2$, arboricity at most $b$, maximum degree $\Delta$, and at
  least $\frac{b\Delta}{4}$ vertices.
\end{proposition}

\begin{proof}
By \lemref{LineGraph}, there is a $(2b-2)$-regular graph $X$ with
  $\binom{b+1}{2}$ vertices, containing cliques $X_1,\dots,X_{b+1}$
  each of order $b$, such that each vertex in $X$ is in exactly two of
  the $X_i$, and $X_i\cap X_j\neq\emptyset$ for all $i,j\in[1,q+1]$.

 Initialise a graph $G$ equal to $X$. For $i\in[1,b+1]$, add an independent set $Y_i$ of
  $p:=\frac{\Delta}{2}-b+1$ vertices to $G$ completely  adjacent to $X_i$.

  Consider two vertices $v$ and $w$ in $G$. Say $v\in X_i\cup Y_i$ and
  $w\in X_j\cup Y_j$. Let $x$ be the vertex in $X_i\cap X_j$. If $v=x$
  or $x=w$ then $vw$ is an edge in $G$, otherwise $vxw$ is a path in
  $G$. Thus $G$ has diameter 2.

  Vertices in each $X_i$ have degree $2b-2+2p=\Delta$ and vertices
  in each $Y_i$ have degree $b\leq\Delta$. Hence $G$ has maximum
  degree $\Delta$.
  The number of vertices in $G$ is more than 
$(b+1)p =(b+1)( \frac{\Delta}{2} -b+1) \geq \frac{b\Delta}{4}$. 

  To calculate the arboricity of $G$, consider a subgraph $H$ of
  $G$. Let $x_i:=|X_i\cap V(H)|$ and $y_i:=|Y_i\cap V(H)|$. Since
  $x_i\leq|X_i|=b$ and $b\geq 2$,
$$
\sum_i \tbinom{x_i}{2} 
=\sum_i \tfrac{x_i(x_i-1)}{2} 
\leq \sum_i\tfrac{b(x_i-1)}{2} 
= \sum_i \tfrac{bx_i-b}{2} 
< (\sum_i \tfrac{bx_i}{2}) -b\enspace.$$ 
Since $x_iy_i\leq|X_i|y_i= by_i$,
$$
\sum_i \tbinom{x_i}{2}+x_iy_i 
\leq (\sum_i \tfrac{bx_i}{2}+by_i) -b 
= b \big((\sum_i \tfrac{x_i}{2}+y_i) -1\big) \enspace.$$ 
Observe that
$|E(H)|\leq \sum_i \binom{x_i}{2}+x_iy_i$ and $|V(H)|\geq
\sum_i\tfrac{x_i}{2}+y_i$ (since each vertex in $X$ is in exactly two
of the $X_i$). Thus $|E(H)| \leq b( |V(H)| - 1 )$, and $G$ has arboricity at most $b$ by
\eqref{NashWilliams}.
\end{proof}





We conclude this section with an open problem about the degree-diameter problem for graphs containing no $K_t$-minor. Every such graph has arboricity at most $ct\sqrt{\log t}$, for some constant $c>0$; see \citep{Kostochka82,Thomason01,Thomason84}. Thus
\thmref{ArboricityUpperBound} implies that for every $K_t$-minor-free graph $G$ with diameter $k$ and maximum  degree $\Delta\gg t$, 
$$|V(G)| \leq 4k (c t\sqrt{\log t} )^k \Delta^{\floor{k/2}}.$$
Improving the  $f(t,k)$ term in this $f(t,k)\,\Delta^{\floor{k/2}}$ bound is a challenging open problem.

\section{Separators and Treewidth}
\seclabel{Separators}

This section studies a separator-based approach for proving upper bounds in the degree-diameter problem. A \emph{separation} of order $s$ in an $n$-vertex graph $G$ is a partition $(A,S,B)$ of $V(G)$, such that $|A|\leq\frac23 n$ and $|B|\leq\frac23 n$ and $|S|\leq s$ and there is no edge between $A$ and $B$. \citet{FHS95} first used separators to prove upper bounds in the degree--diameter problem. In particular, they implicitly proved that every graph that has a separation of order $s$ has $3s\,M(\Delta,\floor{\frac{k}{2}})$ vertices. The following lemma improves the dependence on $s$ in this result when $k$ is even. We include the proof by \citet{FHS95} for completeness.

\begin{lemma}
\lemlabel{NewNewSep}
Let $G$ be a graph with maximum degree at most $\Delta$, and diameter at most $k$. Assume $(A,S,B)$ is a separation of order $s$ in $G$. Then 
\begin{equation*}
|V(G)|\leq \begin{cases}
3s\,M(\Delta,\frac{k-1}{2}) & \text{ if $k$ is odd}\\
\frac{3}{2}\sqrt{s}\,\Delta(\Delta-1)^{k/2-1} \,+\,3s\,M(\Delta,\frac{k}{2}-1) & \text{ if $k$ is even}\enspace.
\end{cases}
\end{equation*}
\end{lemma}

\begin{proof}
Let $n:=|V(G)|$. Note that $|A|\geq n-|B|-s\geq\frac{n}{3}-s$. By symmetry, $|B|\geq \frac{n}{3}-s$. We use this fact repeatedly. 

For $v\in A\cup B$, let $\dist(v,S):=\min\{\dist(v,x):x\in S\}$. 
If $\dist(v,S)\geq\floor{k/2}+1$ for some $v\in A$ and 
$\dist(w,S)\geq\floor{k/2}+1$ for some $w\in B$, then 
$\dist(v,w)\geq2\floor{k/2}+2\geq k+1$, which is a contradiction. 
Hence, without loss of generality, 
$\dist(v,S)\leq\floor{k/2}$ for each $v\in A$. 
By the Moore bound, for each vertex $x\in S$, there are at most $M(\Delta,\floor{k/2})-1$
vertices in $A$ at distance at most $\floor{k/2}$ from $x$. 
Each vertex in $A$ is thus counted. 
Hence $$\frac{n}{3}-s\leq|A|\leq s\,M(\Delta,\floor{k/2})-s\enspace,$$ 
implying $n\leq 3s\,M(\Delta,\floor{k/2})$. 
This proves the result of \citet{FHS95} mentioned above, and proves the case of odd $k$ in the theorem. 

Now assume that $k=2\ell$ is even. 
Suppose on the contrary that $$\frac{n}{3}>\frac{\sqrt{s}}{2}\,\Delta(\Delta-1)^{\ell-1} \,+\,s\,M(\Delta,\ell-1)\enspace.$$
First consider the case in which some vertex in $A$ is at distance at least $\ell+1$ from $S$. 
Thus every vertex in $B$ is at distance at most $\ell-1$ from $S$. By the Moore bound,   
$$s\,M(\Delta,\ell-1) -s \geq
|B| \geq \frac{n}{3}-s > \frac{\sqrt{s}}{2}\,M(\Delta,\ell) \,+\,s\,M(\Delta,\ell-1) -s
\enspace,$$
which is a contradiction.
Now assume that every vertex in $A$ is at distance at most $\ell$ from $S$. 
By symmetry, every vertex in $B$ is at distance at most $\ell$ from $S$. 

Let $A'$ and $B'$ be the subsets of $A$ and $B$ respectively at distance exactly $\ell$ from $S$. 
By the Moore bound,  $|A-A'|\leq s\,M(\Delta,\ell-1)-s$. 
Hence $$|A'|=|A|-|A-A'|\geq \frac{n}{3}-s-s\,M(\Delta,\ell-1)+s>\frac{\sqrt{s}}{2}\,\Delta(\Delta-1)^{\ell-1} \enspace.$$
By symmetry, $|B'|> \frac{\sqrt{s}}{2}\,\Delta(\Delta-1)^{\ell-1}$. 

Let $P:=\{(x,y):x\in A',y\in B'\}$. For each pair $(x,y)\in P$, some vertex $v$ in $S$ is at distance $\ell$ from both $x$ and $y$. Charge  $(x,y)$ to $v$. We now bound the number of pairs in $P$ charged to each vertex $v\in S$. 
Say $v$ has degree $a$ in $A$ and degree $b$ in $B$. Thus $a+b \leq \Delta$. There are at most
$a(\Delta-1)^{\ell-1}$ vertices at distance exactly $\ell$ from $v$ in $A$, and
there at most $b(\Delta-1)^{\ell-1}$ vertices at distance exactly $\ell$ from
$v$ in $B$. Thus the number of pairs charged to $v$ is at most
$$ab(\Delta-1)^{2\ell-2} \leq \quarter(a+b)^2(\Delta-1)^{2\ell-2} \leq \quarter\Delta^2(\Delta-1)^{2\ell-2}\enspace.$$ 
Hence
\begin{align*}
\frac{s}{4}\,\Delta^2(\Delta-1)^{2\ell-2}=\left(\frac{\sqrt{s}}{2}\,\Delta(\Delta-1)^{\ell-1}\right)^2 <|A'|\cdot|B'|=|P|
\leq \frac{s}{4}\Delta^2(\Delta-1)^{2\ell-2}\enspace.
\end{align*}
This contradiction proves that 
$n\leq\frac{3}{2}\sqrt{s}\,\Delta(\Delta-1)^{\ell-1} \,+\,3s\,M(\Delta,\ell-1)$. 
\end{proof}

\lemref{NewNewSep} can be written in the following convenient form.

\begin{lemma}
\lemlabel{NewSep}
For all $\epsilon>0$ there is a constant $c_\epsilon$ such that for every  graph $G$ with maximum degree  $\Delta$, diameter $k$, and a separation of order $s$, 
\begin{equation*}
|V(G)|\leq \begin{cases}
(3+\epsilon)s(\Delta-1)^{(k-1)/2} & \text{ if $k$ is odd and $\Delta\geq c_\epsilon$}\\
(\frac{3}{2}+\epsilon)\sqrt{s}\,(\Delta-1)^{k/2} & \text{ if $k$ is even and $\Delta\geq c_{\epsilon}\sqrt{s}$}
\enspace.
\end{cases}
\end{equation*}
\end{lemma}

\begin{proof}
First consider the the odd $k$ case. For $\Delta\geq\frac{6}{\epsilon}+2$ we have $3(\tfrac{\Delta}{\Delta-2})\leq3+\epsilon$. Thus, by \lemref{NewNewSep} and the Moore bound, 
\begin{equation*}
|V(G)|
\;\leq\;
3s\,(\tfrac{\Delta}{\Delta-2})(\Delta-1)^{(k-1)/2} 
\;\leq\;
(3+\epsilon)s(\Delta-1)^{(k-1)/2}\enspace.
\end{equation*}
Now consider the even $k$ case. For $\Delta\geq\tfrac{3}{\epsilon}+1$ we have
$\tfrac{3}{2}\,\Delta\leq(\tfrac{3}{2}+\tfrac{\epsilon}{2})\,(\Delta-1)$. And for 
$\Delta\geq \tfrac{9}{\epsilon}\sqrt{s}+2$ we have
$3\sqrt{s} \;\leq\; \tfrac{\epsilon}{3}(\Delta-2) \leq \tfrac{\epsilon}{2}(\tfrac{\Delta-1}{\Delta})(\Delta-2)$, implying
$3s\,(\tfrac{\Delta}{\Delta-2}) \leq \tfrac{\epsilon}{2}\sqrt{s}(\Delta-1)$. Hence, by \lemref{NewNewSep} and the Moore bound, 
\begin{align*}
|V(G)|
&\;\leq\;
\tfrac{3}{2}\sqrt{s}\,\Delta(\Delta-1)^{k/2-1}+3s\,(\tfrac{\Delta}{\Delta-2})(\Delta-1)^{k/2-1} \\
&\;\leq\;
(\tfrac{3}{2}+\tfrac{\epsilon}{2})\sqrt{s}\,(\Delta-1)^{k/2}+\tfrac{\epsilon}{2}\sqrt{s}(\Delta-1)^{k/2} \\
&\;\leq\;
(\tfrac{3}{2}+\epsilon)\sqrt{s}\,(\Delta-1)^{k/2}\enspace.\qedhere
\end{align*}
\end{proof}

Treewidth is a key topic when studying separators. In particular, every graph with treewidth $t$ has a separation of order $t+1$, and in fact, a converse result holds~\citep{Reed97}. Thus \lemref{NewSep} implies:

\begin{theorem}
\thmlabel{Treewidth}
For all $\epsilon>0$ there is a constant $c_\epsilon$ such that for every  graph $G$ with maximum degree  $\Delta$, treewidth  $t$, and diameter $k$,
\begin{equation*}
|V(G)|\leq \begin{cases}
(3+\epsilon)(t+1)(\Delta-1)^{(k-1)/2} & \text{ if $k$ is odd and $\Delta\geq c_\epsilon$}\\
(\frac{3}{2}+\epsilon)\sqrt{t+1}\,(\Delta-1)^{k/2} & \text{ if $k$ is even and $\Delta\geq c_{\epsilon}\sqrt{t+1}$}
\enspace.
\end{cases}
\end{equation*}
\end{theorem}

Note that  \thmref{Treewidth} in the case of odd $k$ can also be concluded from a result by \citet[Theorem~3.2]{GPRS01}. Our original contribution is for the even $k$ case. We now show that both upper bounds in \thmref{Treewidth} are within a constant factor of optimal. 

\begin{proposition}
\proplabel{TreewidthConstruction}
For all integers $k\geq 1$ and $t\geq 2$ and $\Delta$ there is a graph $G$ with maximum degree  $\Delta$, diameter $k$, treewidth at most $t$, and  
\begin{equation*}
|V(G)|\geq \begin{cases}
\half(t+1)(\Delta-1)^{(k-1)/2} & \text{ if $k$ is odd and $\Delta\geq 2t-2$}\\
\half\sqrt{t+1}\,(\Delta-1)^{k/2} & \text{ if $k$ is even and $\Delta\geq 4\sqrt{2t}$}\enspace.
\end{cases}
\end{equation*}
\end{proposition}

\begin{proof}
First consider the case of odd $k$. Let $T$ be the rooted tree such that the root vertex has degree $\Delta-t$, every non-root non-leaf vertex has   degree $\Delta$, and the distance between the root and each leaf
  equals $\frac{k-1}{2}$.  Since $t\geq2$ and $\frac{\Delta-t}{\Delta-2}\geq\frac{1}{2}$,
  \begin{align*}
|V(T)|\;=\;   1+(\Delta-t)\sum_{i=0}^{(k-3)/2} (\Delta-1)^i 
& \;=\; \frac{t-2+(\Delta-t)(\Delta-1)^{(k-1)/2}}{\Delta-2}\\
&\;\geq\;  \frac{1}{2}(\Delta-1)^{(k-1)/2}
  \enspace.
  \end{align*}
Take $t+1$ disjoint copies of $T$, and add a clique on their roots. This graph is chordal with maximum clique size $t+1$. Thus it has treewidth $t$. The maximum degree is $\Delta$ and the number of vertices is at least $\half(t+1)(\Delta-1)^{(k-1)/2}$.

Now consider the case of even $k$. Let $q$ be the maximum integer such that $\binom{q+1}{2}\leq t+1$. Thus $2\leq q\leq\sqrt{2t}\leq \frac{\Delta}{4}$ and $q+1\geq\sqrt{t+1}$. 
Let $T$ be the tree, rooted at $r$, such that $r$ has degree $\Delta-q$, every non-leaf non-root vertex has degree $\Delta$, and the distance between $r$ and each leaf is $\frac{k}{2}-1$. Since $q\geq 2$ and $\frac{\Delta-q}{\Delta-2}\geq\half$, 
  \begin{align*}
|V(T)|\;=\;   1+(\Delta-q)\sum_{i=0}^{k/2-2} (\Delta-1)^i 
\;& =\; \frac{q-2+(\Delta-q)(\Delta-1)^{k/2-1}}{\Delta-2}\\
&\geq\;  \frac{1}{2}(\Delta-1)^{k/2-1}
  \enspace.
  \end{align*}
By \lemref{LineGraph},  there is a $(2q-2)$-regular graph $L$ with $\binom{q+1}{2}$ vertices, containing cliques $L_1,\dots,L_{q+1}$ each of order $q$, such that each vertex in $L$   is in exactly two of the $L_i$, and $L_i\cap L_j\neq\emptyset$ for   all $i,j\in[1,q+1]$. Let $G$ be the graph obtained from $L$ as follows. For  each $i\in [1,q+1]$, add $\Delta-2(q-1)$ disjoint copies of $T$ (called $i$-copies), where every vertex in $L_i$ is adjacent to the roots of the $i$-copies of $T$, as illustrated in \figref{EvenTreewidth}. It is easily verified that $G$ has maximum degree $\Delta$. Consider a vertex $v$ in some $i$-copy of $T$ or in $L_i$, and  a vertex $w$ in some $j$-copy of $T$ or in $L_j$. Let $x$ be in $L_i\cap L_j$. Then $\dist(v,x)\leq\frac{k}{2}$ and $\dist(w,x)\leq\frac{k}{2}$, implying $\dist(v,w)\leq k$. Hence $G$ has diameter at most $k$. Let $G'$ be the super graph of $G$ obtained by adding a clique on $V(L)$. Thus $G'$ is chordal with maximum clique size $\binom{q+1}{2}\leq t+1$. Hence $G$ has treewidth at most $t$. 
The number of vertices in $G$ is at least $(q+1)(\Delta-2q+2)|V(T)|\geq\sqrt{t+1}\cdot\frac{\Delta}{2}\cdot(\Delta-1)^{k/2-1}$. 
\end{proof}

\begin{figure}[!h]
\includegraphics{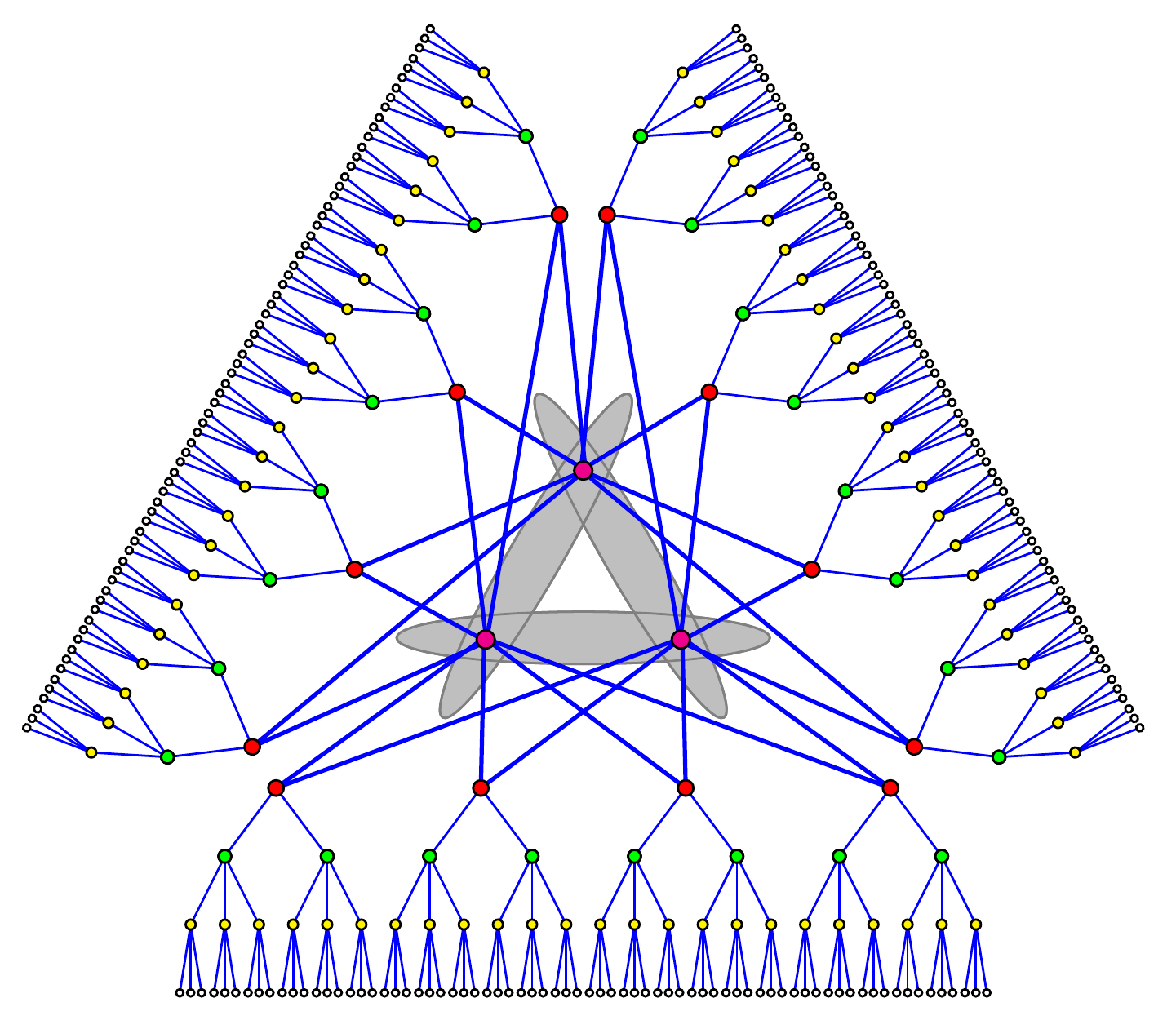}
\caption{\figlabel{EvenTreewidth}Construction in \propref{TreewidthConstruction} for even $k$. Here $\Delta=4$ and $k=8$ and $t=2$.}
\end{figure}


We now consider the degree-diameter problem for graphs with given Euler genus. Note that the case of planar graphs has been widely studied \citep{HS93,FHS95,FHS98,Tis12,Tis12a,NPW}. 
\citet{SS04} proved that for every graph $G$ with Euler genus $g$,
\begin{equation}
  \label{SS04}
|V(G)| \leq c(g+1)k\,(\Delta-1)^{\floor{k/2}}\enspace,
\end{equation}
for some absolute constant $c$. 
 \citet{Eppstein-Algo00} proved that every graph with Euler genus $g$ and diameter $k$ has treewidth at most $c(g+1)k$ for some absolute constant $c$, and Dujmovic et al.~\cite{DujMorWoo13} proved the explicit bound of $(2g+3)k$. \thmref{Treewidth} thus implies the  upper bound in  \eqref{SS04} and improves upon it when $k$ is even:

\begin{theorem}
For all $\epsilon>0$ there is a constant $c_\epsilon$ such that for
every  graph $G$ with Euler genus $g$, maximum degree $\Delta$ and
diameter $k$,
\begin{equation*}
|V(G)|\leq \begin{cases}
(3+\epsilon)((2g+3)k+1)(\Delta-1)^{(k-1)/2} & \text{ if $k$ is odd and
$\Delta\geq c_\epsilon$}\\
(\frac{3}{2}+\epsilon)\sqrt{(2g+3)k+1}\,(\Delta-1)^{k/2} & \text{ if
$k$ is even and $\Delta\geq c_{\epsilon}\sqrt{(2g+3)k+1}$}
\enspace.
\end{cases}
\end{equation*}
\end{theorem}

%


In our companion paper~\citep{NPW} we further investigate the degree-diameter problem for graphs on surfaces, providing an improved upper bound and a new lower bound.

To obtain an upper bound of the form  $n\leq f(k)\,\Delta^{\floor{k/2}}$ using the separator-based approach, one needs a separation of order  bounded by a function of the graph's diameter.  In some sense, the graphs that have a  separation of bounded order are precisely the graphs with bounded treewidth. See Reed's survey \citep{Reed97} for a precise statement here. Thus the separator-based method only works for graphs whose treewidth is  bounded by a function of their diameter. The minor-closed graph classes with this property are precisely those that exclude a fixed apex graph as a minor \citep{Eppstein-Algo00}.  Here a graph $H$ is \emph{apex} if $H-v$ is planar for some vertex $v$ of $H$. For example, $K_5$ and $K_{3,3}$ are apex.   \citet{Eppstein-Algo00} proved that for some function $f$ (depending   on $H$), the treewidth of every $H$-minor-free graph $G$ is at most   $f(\diam(G))$. This is called the \emph{diameter-treewidth} or   \emph{bounded local treewidth} property; also see    \citep{DH-SJDM04,DH-Algo04,Grohe-Comb03,DujMorWoo13}.  \citet{DH-SODA04}   strengthened Eppstein's result by showing that one can take   $f(k)=ck$ for some constant $c=c(H)$. Thus the next result follows from \thmref{Treewidth}. 
  

\begin{theorem}
  \thmlabel{ExcludeApex}
  For every fixed apex graph $H$ there is a constant $c=c(H)$, such that for every $H$-minor-free graph $G$ with diameter $k$,
\begin{equation*}
|V(G)| \leq 
\begin{cases}
ck\,(\Delta-1)^{(k-1)/2}&\text{ if $k$ is odd}\\
c\sqrt{k}\,(\Delta-1)^{k/2}&\text{ if $k$ is even and $\Delta\geq c\sqrt{k}$}\enspace.
\end{cases}
\end{equation*}
\end{theorem}

As  discussed above, for minor-closed classes, \thmref{ExcludeApex} is
the strongest possible result that can be obtained using the separator-based method.

\section{3-Colourable and Triangle-Free Graphs}
\seclabel{3Colour}

As mentioned in the introduction, it is well known that the maximum number of vertices in a bipartite graph is $f(k)\,\Delta^{k-1}$. We now show that this bound does not hold for the more general class of 3-colourable graphs. In fact, we construct 3-colourable graphs where the number of vertices is within a constant factor of the Moore bound. First note that \citet{KS02} proved (building on the work of \citet{HE72}) that for large $k\gtrsim\log r$, the de Bruijn graph $B(r,k)$, which roughly has $\big(\tfrac{\Delta}{2}\big)^k$ vertices, is 3-colourable. The constructions below have the advantage of not assuming that $k$ is large. 

In what follows a \emph{pseudograph} is an undirected graph possibly with loops. A  loop at a vertex $v$ counts for 1 in the degree of $v$. A pseudograph $H$ is \emph{$k$-good} if for all (not necessarily distinct)  vertices $v$ and $w$ there is a $vw$-walk of length exactly $k$ in $H$. 

Given pseudographs $H_1$ and $H_2$, the \emph{direct product} graph $H_1\times H_2$ has vertex set $V(H_1)\times V(H_2)$, where $(v,x)(w,y)\in E(H_1\times H_2)$ if and only if $vw\in E(H_1)$ and $xy\in E(H_2)$. 

\begin{lemma}
\lemlabel{Product}
Let $H_1$ and $H_2$ be $k$-good pseudographs with maximum degree $\Delta_1$ and $\Delta_2$ respectively. Then $H_1\times H_2$ has $|V(H_1)|\cdot |V(H_2)|$ vertices, maximum degree $\Delta_1\Delta_2$, and diameter at most $k$. Moreover, if $H_2$ is loopless and $c$-colourable, then $H_1\times H_2$ is $c$-colourable.
\end{lemma}

\begin{proof}
Clearly  $H_1\times H_2$ has $|V(H_1)|\cdot |V(H_2)|$ vertices and maximum degree $\Delta_1\,\Delta_2$. 
Let $(v,x)$ and $(w,y)$ be distinct vertices of $G$. To prove that $G$ has diameter at most $k$, we  construct a $(v,x)(w,y)$-walk of length at most $k$ in $G$. Since $H_1$ is $k$-good, there is a walk $v=v_0,v_1,\dots,v_k=w$ of length  $k$ in $H_1$. Since $H_2$ is $k$-good, there is a walk $x=x_0,x_1,\dots,x_k=y$ of length  $k$ in $H_2$. Thus $(v,x)=(v_0,x_0),(v_1,x_1),\dots,(v_k,x_k)=(w,y)$ is a walk of length $k$ between $(v,x)$ and $(w,y)$ in $H_1\times H_2$. Hence $H_1\times H_2$ has diameter at most $k$. Finally, colouring each vertex $(v,x)$ of $H_1\times H_2$ by the colour assigned to $x$ in a $c$-colouring of $H_2$ gives a $c$-colouring of $H_1\times H_2$. 
\end{proof}

\begin{lemma} 
\lemlabel{K3}
$K_3$ is $k$-good for all $k\geq 2$.
\end{lemma}

\begin{proof} 
Let $v,w\in V(K_3)=\{0,1,2\}$. If there is a $vw$-walk of length $k-2$, then there is a $vw$-walk of length $k$ (just repeat one edge twice). Thus the claim follows from the $k=2$ and $k=3$ cases. Without loss of generality, $v=0$. For $k=2$, one of $010$, $021$ and $012$ is a $vw$-walk of length 2. For $k=3$, one of $0120$, $0121$ and  $0102$  is a $vw$-walk of length 3.
\end{proof}

To obtain results for triangle-free graphs we use the following:

\begin{lemma} 
\lemlabel{C5}
$C_5$ is $k$-good for all $k\geq 4$.
\end{lemma}

\begin{proof} 
Say $V(C_5)=\{0,1,2,3,4\}$ and $E(C_5)=\{01,12,23,34,40\}$. Let $v,w\in V(C_5)$. If there is a $vw$-walk of length $k-2$, then there is a $vw$-walk of length $k$ (just repeat one edge twice). Thus the claim follows from the $k=4$ and $k=5$ cases. Without loss of generality, $v=0$. For $k=4$, one of $01010$, $04321$, $01212$, $04343$ and $01234$ is a $vw$-walk of length 4. For $k=5$, one of 
$012340$, $040101$, $043232$, $012323$ and $010404$ is a $vw$-walk of length 5.
\end{proof}

\twolemref{Product}{K3} imply:

\begin{lemma}
\lemlabel{ThreeColourable}
Let $H$ be a $k$-good pseudograph with maximum degree $\Delta$ for some $k\geq 2$. Then $H\times K_3$ is a 3-colourable graph with $3|V(H)|$ vertices, maximum degree $2\Delta$, and diameter at most $k$.
\end{lemma}

\begin{lemma}
\lemlabel{TriangleFree}
Let $H$ be a $k$-good pseudograph  with maximum degree $\Delta$ for some $k\geq 4$. Then $H\times C_5$ is a 3-colourable triangle-free graph with $5|V(H)|$ vertices, maximum degree $2\Delta$, and diameter at most $k$.
\end{lemma}

\begin{proof} For any graph $G$ (without loops), if $H\times G$ contains a triangle $(a,u)(b,v)(c,w)$, then $uvw$ is a triangle in $G$ (even if $H$ has loops). Since $C_5$ is triangle-free, $H\times C_5$ is triangle-free. Thus \twolemref{Product}{C5} imply the claim.
\end{proof}

For particular values of $\Delta$ and $k$, various constructions for the degree-diameter problem can be used in the following lemma to give large 3-colourable and triangle-free graphs. 

\begin{proposition}
Let $H$ be a graph with maximum degree $\Delta$ and diameter $k\geq 2$. Then there is a 3-colourable graph with $3|V(H)|$ vertices,  maximum degree $2\Delta+2$ and diameter at most $k$. Moreover, if $k\geq 4$ then there is a 3-colourable triangle-free graph with $5|V(H)|$ vertices, maximum degree $2\Delta+2$, and diameter at most $k$.
\end{proposition}

\begin{proof}
Let $H'$ be the pseudograph obtained from $H$ by adding a loop at each vertex. Thus $H'$ is $k$-good and has maximum degree $\Delta+1$. \twolemref{ThreeColourable}{TriangleFree} imply that $H'\times K_3$ and $H'\times C_5$ satisfy the claims. 
\end{proof}

This result implies that for fixed $k\geq2$ and $\Delta\gg k$, the maximum number of vertices in a 3-colourable graph is within a constant factor of the unrestricted case. And the same conclusion holds for $k\geq4$ for 3-colourable triangle-free graphs. 

We now give  a concrete example:

\begin{theorem}
For all integers $\Delta\geq 4$ and $k\geq 2$, there is a 3-colourable graph with $3\floor{\frac{\Delta}{4}}^{k}$ vertices, maximum degree at most $\Delta$, and diameter at most $k$. Moreover, if $k\geq 4$ then there is a 3-colourable triangle-free graph with $5\floor{\frac{\Delta}{4}}^{k}$ vertices, maximum degree at most $\Delta$, and diameter at most $k$. 
\end{theorem}

\begin{proof}
Let $r:=\floor{\frac{\Delta}{4}}$. Let $H$ be the undirected pseudograph underlying the de Bruijn digraph $\overrightarrow{B}(r,k)$ including any loops. \lemref{deBruijn} shows that $H$ has $r^k$ vertices, maximum degree at most $2r$, and is $k$-good. \lemref{ThreeColourable} shows that $H\times K_3$ satisfies the first claim. \lemref{TriangleFree} implies  that $H\times C_5$ satisfies the second claim.
\end{proof}

\journalarxiv{In the expanded version of this paper \citep{PW13}, we give ad-hoc constructions of triangle-free graphs with diameter 2 and 3, where the number of vertices is $\Omega(\Delta^2)$ and $\Omega(\Delta^3)$ respectively, which is within a constant factor of the Moore bound.}{We now give ad-hoc constructions of triangle-free graphs with diameter 2 and 3. These lower bounds are within a constant factor of the Moore bound. Let $\mathbb{Z}_p$ be the cyclic group with $p$ elements. For $a,b\in\mathbb{Z}_p$, let $\dist(a,b):=\min\{a-b,b-a\}$. Here, as always, addition is in the group. 

\begin{proposition}
\proplabel{TriangleFreeDiameter2}
For all $\Delta\geq 20$ there is a triangle-free graph with diameter $2$, maximum degree at most  $\Delta$, and at least $(2\floor{\frac{\Delta+4}{8}}+2)^2$ vertices.
\end{proposition}

\begin{proof}
Let $p:=2\floor{\frac{\Delta+4}{8}}+2$. Thus $p\geq 8$ is even. 
Let $G$ be a graph with vertex set $\mathbb{Z}_p^2$. Thus $|V(G)|=(2\floor{\frac{\Delta+4}{8}}+2)^2$. 
Let $(v_1,v_2)$ denote a vertex $v$ in $G$. For distinct vertices $v$ and $w$, define the \emph{$vw$-vector} to be $(a,b)$, where  $a \leq b$ and $\{a,b\} = \{ \dist(v_1,w_1), \dist(v_2,w_2) \}$. Then $vw \in E(G)$ if and only if $a=1$ and $b\neq 2$. Observe that $G$ is $4(p-3)$-regular, and  $4(p-3)\leq \Delta$.

We now show that the distance between distinct vertices $v,w$ in $G$ is at most 2. Consider the following cases for the $vw$-vector ($a,b)$, where without loss of generality, $(a,b) = ( \dist(v_1,w_1), \dist(v_2,w_2) )$:

Case $(0,\geq 1)$: Since $p\geq 8$, there exists $y\in\mathbb{Z}_p$ 
such that $\dist(v_2,y)\not\in\{0,2\}$ and $\dist(w_2,y)\not\in\{0,2\}$.  
Then $(v_1+1,y)=(w_1+1,y)$ is a common neighbour of $v$ and $w$.

Case $(1, 2)$: Since $p\geq 8$, there exists $x\in\mathbb{Z}_p$ such that $\dist(v_1,x)\not\in\{0,2\}$ and $\dist(w_1,x)\not\in\{0,2\}$.  Since $\dist(v_2,w_2)=2$ there exists $y\in\mathbb{Z}_p$ such that $\dist(v_2,y)=\dist(w_2,y)=1$. Then $(x,y)$ is a common neighbour of $v$ and $w$.

Case $(1,\neq 2)$: Then $v$ and $w$ are adjacent.

Case $(\geq 2, \geq 2)$: Since $p\geq 8$, there exists $x\in\mathbb{Z}_p$ such that $\dist(w_1,x)=1$ and $\dist(v_1,x)\not\in\{0,2\}$.  Similarly, there exists $y\in\mathbb{Z}_p$ such that $\dist(y,v_1)=1$ and $\dist(w_2,y)\not\in\{0,2\}$.  Then $(x,y)$ is a common neighbour of $v$ and $w$.

%
%

Suppose on the contrary that $G$ contains a triangle $T$. For each edge $uv$ of $T$, we have $\dist(u_i,v_i)=1$ for some $i\in[1,2]$. In this case, say $uv$ is \emph{type} $i$. Since there are three pairs of vertices in $T$ and only two types, two pairs of vertices in $T$ have the same type. Say $T=uvw$. Without loss of generality, $uv$ and $vw$ are both type-1. That is, 
$\dist(u_1,v_1)=1$ and $\dist(v_1,w_1)=1$. Thus $\dist(u_1,w_1)\in\{0,2\}$, in which case  $uw\not\in E(G)$. This contradiction shows that $G$ contains no triangle. 
\end{proof}

\begin{proposition}
\proplabel{TriangleFreeDiameter3}
For all $\Delta\geq 42$ there is a  triangle-free graph with diameter $3$, maximum degree at most $\Delta$, and at least $(2\floor{\frac{\Delta+6}{12}}+4)^3$ vertices.
\end{proposition}

\begin{proof}
Let $p:=2\floor{\frac{\Delta+6}{12}}+4$. Thus $p\geq 12$ is even. Let $H$ be the graph with vertex set $\mathbb{Z}_p$, where $ab\in E(H)$ whenever $\dist(a,b)\geq 3$. 
Observe that every pair of vertices in $H$ have a common neighbour (since $p\geq 12$). 

Define a graph $G$ with vertex set $V(G) :=\mathbb{Z}_p^3$. Thus $|V(G)|=p^3$. Let $(v_1,v_2,v_3)$ denote a vertex $v$ in $G$. For distinct vertices $v$ and $w$, define the \emph{$vw$-vector} to be $(a,b,c)$, where  $a \leq b\leq c$ and $\{a,b,c\} = \{ \dist(v_1,w_1), \dist(v_2,w_2), \dist(v_3,w_3) \}$. Then $vw \in E(G)$ if and only if $a=0$ and $b=1$ and $c \geq 3$. Observe that $G$ is $6(p-5)$-regular, and  $6(p-5)\leq \Delta$.

We now show that the distance between distinct vertices $v,w$ in $G$ is at most 3. Consider the following cases for the $vw$-vector, where without loss of generality, $(a,b,c) = ( \dist(v_1,w_1), \dist(v_2,w_2), \dist(v_3,w_3) )$:

Case $(0,0,\geq 1)$: 
Let $u$ be a common neighbour of $v_3$ and $w_3$ in $H$. 
Then $(v_1+1,v_2,u)=(w_1+1,w_2,u)$ is a common neighbour of $v$ and $w$.

Case $(0,1,1)$: Then $(v_1+3,w_2,v_3)=(w_1+3,w_2,v_3)$ is a common neighbour of $v$ and $w$. 

Case $(0,1,2)$: Let $y$ be a common neighbour of $v_2$ and $w_2$ in $H$. Since $\dist(v_3,w_3)=2$, there is an element $z$ such that $\dist(v_3,z)=\dist(w_3,z)=1$. Then $(v_1,y,z)=(w_1,y,z)$ is a common neighbour of $v$ and $w$. 

Case $(0,1,\geq 3)$: Then $v$ and $w$ are adjacent.

Case $(0,\geq 2,\geq 2)$: Since $\dist(v_2, w_2)\geq 2$, there is an element $y\in\{w_2-1,w_2+1\}$ such that $\dist(v_2,y)\geq 3$. 
Similarly, $\dist(w_3,z)\geq 3$ for some  $z\in\{v_3-1,v_3+1\}$. Then $(v_1,y,z)$ is a common neighbour of $v$ and $w$. 

Case $(1,\geq 1,\geq 1)$: Since $v_2\neq w_2$, there is an element $u\in\{w_2,w_2+2,w_2-2\}$ such that $\dist(v_2,u)\geq 3$. Let $u'$ be such that $\dist(u,u')=\dist(w_2,u')=1$. Let $z$ be a common neighbour of $v_3$ and $z_3$ in $H$. Then $$(v_1,v_2,v_3)(w_1,u,v_3)(w_1,u',z)(w_1,w_2,w_3)$$ is a $vw$-path of length 3.  

Case $(\geq 2, \geq 2, \geq 2)$: Since $p\geq 6$ and $\dist(v_1,w_1)\geq2$, there exists  $x\in\mathbb{Z}_p$ such that $\dist(w_1,x)=1$ and $\dist(v_1,x)\geq 3)$. Similarly, there exists  $y,z\in\mathbb{Z}_p$ such that $\dist(w_2,y)=1$ and $\dist(v_2,y)\geq 3$, and  $\dist(v_3,z)=1$ and $\dist(w_3,z)\geq 3$.  Then $$(v_1,v_2,v_3)(x,v_2,z)(w_1,y,z)(w_1,w_2,w_3)$$ is a $vw$-path of length $3$.

Thus $G$ has diameter at most 3. 

Suppose on the contrary that $G$ contains a triangle $T$. For each edge $uv$ of $T$, we have $u_i=v_i$ for exactly one value of  $i\in[1,3]$. In this case, say $uv$ is \emph{type} $i$. First suppose that at least two of the edges in $T$ are the same type. Then all three edges in $T$ are the same type. Without loss of generality, $u_1=v_1=w_1$. Then the subgraph of $G$ induced by $\{u,v,w\}$ (ignoring the first coordinate) is a subgraph of the graph in \propref{TriangleFreeDiameter2}, which is triangle-free. Now assume that all three edges in $T$ have distinct types. Without loss of generality, $u_1=v_1$ and $u_2=w_2$ and $v_3=w_3$. Since $uv\in E(G)$, without loss of generality, $\dist(u_2,v_2)=1$ and $\dist(u_3,v_3)\geq 3$. Thus $\dist(v_2,w_2)=1$ and $\dist(u_3,w_3)\geq 3$. Since $vw\in E(G)$ and $u_1=v_1$, we have $\dist(u_1,w_1)=\dist(v_1,w_1)\geq 3$. We have shown that $\dist(u_1,w_1)\geq 3$ and $u_2=w_2$ and $\dist(u_3,w_3)\geq 3$. Thus the $uw$-vector is $(0,3,3)$, implying $uw\not\in E(G)$. This contradiction shows that $G$ is triangle-free. 
\end{proof}

Finally, note that the graphs in \twopropref{TriangleFreeDiameter2}{TriangleFreeDiameter3} have bounded chromatic number. In \propref{TriangleFreeDiameter2}, colour each vertex $v$ by $(v_1 \bmod 2, v_2 \bmod 2)$. For each edge $vw$, we have $\dist(v_i,w_i)=1$ for some $i$. Since $p$ is even, $v_i\not\equiv w_i\pmod{2}$. Thus this is a valid 4-colouring. In 
 \propref{TriangleFreeDiameter3},  colouring each vertex $v$ by $(v_1 \bmod 2, v_2 \bmod 2,v_3\bmod 2)$ gives an 8-colouring.
} 

\section*{Acknowledgement} The case $k=2$ of \thmref{Treewidth} was proved in collaboration with Bruce Reed. Thanks Bruce. 


\def\soft#1{\leavevmode\setbox0=\hbox{h}\dimen7=\ht0\advance \dimen7
  by-1ex\relax\if t#1\relax\rlap{\raise.6\dimen7
  \hbox{\kern.3ex\char'47}}#1\relax\else\if T#1\relax
  \rlap{\raise.5\dimen7\hbox{\kern1.3ex\char'47}}#1\relax \else\if
  d#1\relax\rlap{\raise.5\dimen7\hbox{\kern.9ex \char'47}}#1\relax\else\if
  D#1\relax\rlap{\raise.5\dimen7 \hbox{\kern1.4ex\char'47}}#1\relax\else\if
  l#1\relax \rlap{\raise.5\dimen7\hbox{\kern.4ex\char'47}}#1\relax \else\if
  L#1\relax\rlap{\raise.5\dimen7\hbox{\kern.7ex
  \char'47}}#1\relax\else\message{accent \string\soft \space #1 not
  defined!}#1\relax\fi\fi\fi\fi\fi\fi}

\end{document}